\documentclass[seceqn,MSNbibl,number,citesort,dvips]{arxbj}
\usepackage[ruled]{algorithm2e}

\aid{0}
\volume{20}
\issue{4}
\pubyear{2014}
\firstpage{2217}
\lastpage{2246}
\doi{10.3150/13-BEJ555} 

\makeatletter

\let\my@algocf@latexcaption\algocf@latexcaption
\let\my@addcontentsline\addcontentsline
\long\def\algocf@latexcaption#1[#2]#3{%
\def\addcontentsline##1##2##3{}%
\my@algocf@latexcaption{#1}[#2]{#3}%
\global\let\addcontentsline\my@addcontentsline%
}

\newcommand{\rright}{\right}
\newcommand{\lleft}{\left}
\newcommand{\rrvert}{\vert}
\renewcommand{\vec}[1]{{\mathbf{#1}}}
\newcommand{\supp}{\operatorname{supp}}

\def\1{\mathbf{1}}
\def\E{\mathbb{E}}
\def\V{\mathbb{V}}
\def\P{\mathbb{P}}
\def\R{\mathbb{R}}
\def\N{\mathbb{N}}
\def\C{\mathcal{C}}
\def\LR{\operatorname{LR}}
\def\perm{\operatorname{perm}}

\newcommand{\TV}{\operatorname{TV}}
\newcommand{\KL}{\operatorname{KL}}
\newcommand{\eqref}[1]{(\ref{#1})}
\newcommand{\DS}{{\mathrm{DS}}}
\newcommand{\FDR}{\operatorname{FDR}}
\newcommand{\NDR}{\operatorname{NDR}}
\renewcommand{\perm}{\operatorname{perm}}
\newcommand{\Unif}{\operatorname{Unif}}

\renewcommand{\emptyset}{\varnothing}
\renewcommand{\epsilon}{\varepsilon}

\newtheorem{theorem}{Theorem}[section]
\newtheorem{lemma}{Lemma}[section]
\newtheorem{coro}{Corollary}[section]
\newproclaim{defn}{Definition}[section]
\newtheorem{prop}{Proposition}[section]
\newremark{remark}{Remark}[section]
\makeatother

\begin{document}
\begin{frontmatter}

\title{Adaptive sensing performance lower bounds for sparse signal
detection and support estimation}
\runtitle{Adaptive sensing lower bounds for sparse signal estimation and testing}

\begin{aug}
\author{\inits{R.M.}\fnms{Rui M.} \snm{Castro}\ead[label=e1]{rmcastro@tue.nl}}
\address{Eindhoven University of Technology, The Netherlands.
\printead{e1}}
\dedicated{In memory of Yuri Ingster}
\end{aug}

\received{\smonth{6} \syear{2012}}
\revised{\smonth{3} \syear{2013}}

%
\begin{abstract}
This paper gives a precise characterization of the fundamental limits
of adaptive sensing
for diverse estimation and testing problems concerning sparse signals.
We consider in
particular the setting introduced in (\textit{IEEE Trans. Inform. Theory} \textbf{57} (2011) 6222--6235) and show
necessary conditions on
the minimum signal magnitude for both detection and estimation: if
$\vec{x}\in\R^n$ is a sparse
vector with $s$ non-zero components then it can be reliably detected in
noise provided the
magnitude of the non-zero components exceeds $\sqrt{2/s}$. Furthermore,
the signal support can
be exactly identified provided the minimum magnitude exceeds $\sqrt {2\log s}$. Notably there
is no dependence on $n$, the extrinsic signal dimension. These results
show that the adaptive
sensing methodologies proposed previously in the literature are
essentially optimal, and
cannot be substantially improved. In addition, these results provide
further insights on the limits of adaptive compressive sensing.
\end{abstract}

%

\end{frontmatter}

\section{Introduction}\label{sec:introduction}

This paper addresses the characterization of the fundamental limits of
adaptive sensing in sparse settings, when a potentially infinite number
of observations is available but there is a restriction on the sensing
precision budget available. One of the key aspects of adaptive sensing
is that the data collection process is sequential and adaptive. In
different fields these sensing/experimenting paradigms are known by
different names, such as \emph{sequential experimental design} in
statistics and economics (see Wald \cite{wald:47}, Bessler \cite{bessler:60},
Fedorov \cite{fedorov:72},
El-Gamal \cite{elgamal:91}, Hall and Molchanov \cite{hall:03}, Lai and Robbins \cite{lai:85},
Blanchard and Geman \cite{blanchard:05}),
\emph{active learning} or \emph{adaptive sensing/sampling} in
computer science, engineering and machine learning (see Cohn, Ghahramani and Jordan \cite
{cohn:96}, Freund \textit{et al.} \cite{freund:97}, Novak \cite{novak:96},
Korostelev and Kim \cite{korostelev:00}, Dasgupta \cite{dasgupta:04},
Castro, Willett and Nowak \cite{castro:05}, Dasgupta, Kalai and Monteleoni \cite{dasgupta:05a},
Dasgupta \cite{dasgupta:05b}, Hanneke \cite{hanneke:10},
Koltchiinskii \cite{koltchiinskii:10}, Balcan, Beygelzimer and Langford \cite{balcan:06},
Castro and Nowak \cite{castro:08}).

The extra flexibility of adaptive sensing can sometimes (but not
always) yield significant performance gains. In this paper, we are
particularly concerned with the setting introduced in Haupt, Castro and Nowak \cite{haupt:10},
where the authors propose an adaptive sparse signal recovery method
that provably improves on the best possible non-adaptive sensing
methods. However, in that work there is no indication on the
fundamental performance limitations in such\vadjust{\goodbreak} sensing scenarios. This
paper addresses those breeches in our understanding, and shows that the
proposed procedures are essentially asymptotically optimal for
estimation problems. Furthermore, with some modifications, the
procedure of Haupt, Castro and Nowak \cite{haupt:10} is also nearly
optimal when testing for
the presence of a sparse signal. In addition, we also present results
characterizing the fundamental limitations in several other settings,
such as exact support recovery, as in Malloy and Nowak \cite{malloy:11}, Malloy and Nowak \cite{malloy:11b} or
in Arias-Castro, Cand\`{e}s and Davenport \cite{arias-castro:11}.

\section{Problem setting}

Let $\vec{x}\in\R^n$ be an unknown vector. We assume this vector is
sparse in the sense that only a reduced number of its entries are
not-zero. In particular, let $S$ be a subset of $\{1,\ldots,n\}$ and
assume that for all $i\in\{1,\ldots,n\}$ such that $i\notin S$ we
have $x_i=0$. We refer to $S$ as the signal support set and this is our
main object of interest. In this paper, we consider two distinct
classes of problems: (i) signal support estimation, where we desire to
estimate $S$; (ii) signal detection, where we simply want to test if
$S$ belongs to some particular class.

In our model the signal $\vec{x}$ is unknown, but we can collect
partial information through noisy observations. In particular, we observe
%
\begin{equation}\label{eqn:observations}
Y_k=x_{A_k}+\Gamma^{-1}_k
W_k ,\qquad  k=1,2,\ldots,
\end{equation}
where $A_k,\Gamma_k$ are taken to be measurable functions of $\{
Y_i,A_i,\Gamma_i\}_{i=1}^{k-1}$, and $W_k$ are standard normal random
variables, independent of $\{Y_i\}_{i=1}^{k-1}$ and also independent of
$\{A_i,\Gamma_i\}_{i=1}^{k}$. In this model, $A_k\in\{1,\ldots,n\}$
corresponds to the entry of $\vec{x}$, that is, measured at time $k$,
therefore $A_k$ can be viewed as the \emph{sensing action} taken at
time $k$. Similarly, $\Gamma^2_k$ is the \emph{precision} of the
measurement taken at time $k$. Finally, there is a total sensing budget
constraint that must be satisfied, namely
%
\begin{equation}
\sum_{k=1}^\infty\Gamma^2_k
\leq m ,\label{eqn:budget}
\end{equation}
where $m>0$. It is important to note that we can consider both
deterministic sequential designs or random sequential designs. In the
latter, we allow the choices $A_k$ and $\Gamma_k$ to incorporate
extraneous randomness, which is not explicitly described in the model.
Besides being more general this extra flexibility often facilitates the
analysis. The collection of conditional distributions of $A_k,\Gamma
_k$ given $\{Y_i,A_i,\Gamma_i\}_{i=1}^{k-1}$ for all $k$ is referred
to as the \emph{sensing strategy}, and denoted by $\mathcal{A}$. Note
that, within the sensing model above, we can also consider non-adaptive
sensing frameworks, meaning the choice of sensing actions and precision
allocation must be made before collecting any data. Formally, this
means that $\{A_k,\Gamma_k\}_{k\in\N}$ is statistically independent
from $\{Y_k\}_{k\in\N}$. Note that a non-adaptive design can still be random.

The case $m=n$ is of particular interest and this is often considered
in literature as it allows direct comparison between adaptive and
non-adaptive sensing methodologies. If $m=n$ we allow, on average, one
unit of precision for each one of the $n$ signal entries. Therefore if
we assume the signal $\vec{x}$ belongs a class for which there is no
reason to give a priori preference to any particular signal entry the
optimal non-adaptive sensing strategy amounts to measuring each vector
entry exactly once, with precision one.\footnote{Due to statistical
sufficiency there is no gain in measuring each signal entry more than
once.} This is obviously the classical normal means model.

In the following sections, we consider two different scenarios: signal
detection/testing and signal estimation. In both cases, the extra
flexibility of adaptive sensing is shown to be extremely rewarding. We
characterize the fundamental performance limits of adaptive sensing in
those scenarios and show that these limits can be achieved by practical
inference methodologies.

\section{Signal detection}

In this setting, we are interested in a binary hypothesis testing
problem, where we test a simple null hypothesis against a composite
alternative. In particular, the null hypothesis $H_0$ is simply
$S=\emptyset$, and the alternative hypothesis $H_1$ is $S\in\C$,
where $\C$ is some class of non-empty subsets of $\{1,\ldots,n\}$. We
are particularly interested in the case when under the alternative
$H_1$ all the sets in $\C$ have cardinality $s$, meaning that for all
$S\in\C$ we have $|S|=s$. We will consider only such classes as this
greatly simplifies the presentation and is not, for the most part, a
restrictive condition.

Define
\[
x_{\min}=\min \bigl\{|x_i|\dvtx x_i\neq0 , i\in\{1,
\ldots,n\} \bigr\} .
\]
In the following, we characterize the fundamental signal detection
limits, in particular identifying conditions on $x_{\min}$ as a
function of $\C$ and $n$, such that no procedure is able to reliably
distinguish the two hypotheses. Furthermore, these bounds are
essentially tight, in the sense that there exist practical procedures
matching them. For simplicity, we consider only non-negative signals,
meaning that $x_i\geq0$ for all $i\in\{1,\ldots,n\}$. This greatly
simplifies the analysis, without hindering the generality of the
results. More comments about this are issued in Remark~\ref
{rmk:tightness}. Furthermore, the hardest signals to detect or estimate
are of the form
%
\begin{equation}
\label{eqn:simple_model} x_i=\lleft\{ %
\begin{array} {l@{\qquad}l} \mu& \mbox{if } i\in S,
\\
0 & \mbox{otherwise}. \end{array} %
\rright.
\end{equation}
This means that we can restrict our analysis to signals of the form
above, which are entirely described by the signal support set $S$ and
signal amplitude $\mu$. This is also the class of signals considered
in Addario-Berry \textit{et al.} \cite{addario-berry:10} or in
Donoho and Jin \cite{donoho:04} in a non-adaptive
sensing context.

Let
\[
D=\{Y_i,A_i,\Gamma_i\}_{i\in\N}
,
\]
and let $d=\{y_i,a_i,\gamma_i\}_{i=1}^{\infty}$ be a particular
realization of the experimental procedure. Let $\mathcal{A}$ denote a
particular sensing strategy, and $\hat\Phi(D)\in\{0,1\}$ be an
arbitrary testing function, taking the value $1$ if the null hypothesis
is to be rejected, and zero otherwise. For notational convenience we
write simply $\hat\Phi$ where the hat indicates the dependency on the
data $D$. The \emph{risk} of this procedure is given by
\[
R(\hat\Phi)=\P_\emptyset(\hat\Phi\neq0)+\max_{S\in\C}
\P _S(\hat\Phi\neq1) ,
\]
where $\P_S$ denotes the joint probability distribution of $\{
Y_i,A_i,\Gamma_i\}_{i=1}^{\infty}$ for a given value of $S$. Likewise
we use $\E_S$ to denote expectation under $\P_S$.

Now define
%
\begin{equation}
\label{eqn:minimax} c(\mu,\C) = \inf_{\hat\Phi,\mathcal{A}} R(\hat\Phi)=\inf
_{\hat
\Phi,\mathcal{A}} \Bigl\{\P_\emptyset(\hat\Phi\neq0)+\max
_{S\in\C} \P_S(\hat\Phi\neq1) \Bigr\} .
\end{equation}
Our formal goal is to identify the values of the signal magnitude $\mu
$ for which we have necessarily $c(\mu,\C)\geq\epsilon$ for
$\epsilon>0$.
%
\begin{remark}
The choice of risk above is obviously not the only one possible, and in
the literature other choices of risk have been considered, such as
%
\begin{equation}\label{eqn:max_risk}
\tilde R(\hat\Phi)=\max \Bigl\{\P_\emptyset(\hat\Phi\neq0),\max
_{S\in\C} \P_S(\hat\Phi\neq1) \Bigr\}
,
\end{equation}
or
%
\begin{equation}\label{eqn:bayes_risk}
\bar R(\hat\Phi)=\P_\emptyset(\hat\Phi\neq0)+\frac{1}{N}\sum
_{S\in\C} \P_S(\hat\Phi\neq1)
.
\end{equation}
As discussed in Addario-Berry \textit{et al.}  \cite
{addario-berry:10}, the latter measure of risk
corresponds to the view that, under the alternative hypothesis, a set
$S\in\C$ is selected uniformly at random from $\C$. Clearly
\[
\bar R(\hat\Phi) \leq R(\hat\Phi) \leq2\tilde R(\hat\Phi) \leq2 R(\hat
\Phi) .
\]
If there is sufficient symmetry in $\C$ and $\hat\Phi$ these three
risk measures are essentially identical. Whenever possible we
characterize the fundamental limits of adaptive sensing for each one of
the risk measures, but focus primarily on $R(\hat\Phi)$.
\end{remark}

\subsection{Main results -- Detection}

The class $\C$ of all subsets of $\{1,\ldots,n\}$ with cardinality
$s$ is one of particular interest. This is the class of maximal size,
and obviously the one for which we expect the worst performance for
detection. Perhaps surprisingly, under the adaptive sensing paradigm,
the exact same performance lower bound is obtained for \emph{any}
class $\C$ exhibiting some very mild symmetry. This means that, in
many situations, the structure of the class $\C$ does not really help
under the adaptive sensing scenario. This is in stark contrast with
non-adaptive sensing scenarios, where the structure of the set $\C$
can play a very prominent role, as well documented in Addario-Berry \textit{et al.} \cite{addario-berry:10},
Arias-Castro \textit{et al.} \cite{arias-castro:08}, Butucea and Ingster \cite
{butucea:11}. To state the main result
of this section, we need the following definitions:
%
\begin{defn}[(Symmetric class/full range)]
Let $\Xi=\bigcup_{S\in\C} S$ and $S$ be drawn uniformly at random
from $\C$. If for all $i\in\Xi$ we have $\P(i\in S)=s/|\Xi|$ the
class $\C$ is said to be symmetric. Furthermore, if $|\Xi|=n$ the
class is said to be full range.
\end{defn}

It is remarkable that many classes $\C$ of interest satisfy this mild
symmetry, as for instance, all the classes in Addario-Berry \textit{et al.} \cite{addario-berry:10}.
%
\begin{theorem}\label{thm:detection}
Consider the setting above and let $\C$ be a symmetric class. Let
$\hat\Phi(D)$ be an arbitrary testing procedure, where $D=\{
Y_i,A_i,\Gamma_i\}_{i\in\N}$. Finally, let $0<\epsilon<1$ be
arbitrary. If $R(\hat\Phi)\leq\epsilon$ then necessarily
%
\begin{equation}\label{eqn:detection}
x_{\min} \geq\sqrt{\frac{2|\Xi|}{sm} \log\frac{1}{2\epsilon}}
.
\end{equation}
\end{theorem}

This result gives a condition on the minimal signal magnitude necessary
to ensure the detection risk is not too large. Perhaps surprisingly the
lower bound does not include any factor involving specific structural
properties of $\C$, but only the range and cardinality of the
corresponding sets. A possible way to understand this comes from the
following observation: for detection, it suffices to identify a \emph
{single} element of $S$, and there is no need to identify all the
elements. Therefore, cues provided by the structure are not very
informative. In addition, note that the above theorem also applies to
non-symmetric classes provided they contain a symmetric class. Before
proving this result, it is interesting to present a simple corollary
for the case of full range classes, emphasizing the asymptotic behavior.
%
\begin{coro}\label{coro:detection}
Let $\C$ be a symmetric and full range class of sets with cardinality
$s$, where $s$ can be a function of $n$ (this dependence is not
explicitly stated). Let $\hat\Phi_n$ be an arbitrary adaptive sensing
testing procedure. If
\[
\lim_{n\rightarrow\infty} R(\hat\Phi_n)=0
\]
then necessarily
\[
x_{\min} \geq\omega_n \sqrt{\frac{n}{sm}} ,
\]
where $\omega_n$ is a sequence for which $\lim_{n\rightarrow\infty}
\omega_n=\infty$.
\end{coro}

This corollary gives a necessary condition for detection consistency.
As shown in Proposition~\ref{prop:MDS}, this bound is actually tight,
meaning there are adaptive sensing procedures that can detect signals
satisfying the above condition. The case $m=n$ is particularly
interesting, as it allows the comparison between adaptive and
non-adaptive sensing performance. For that case, adaptive sensing
detection is possible if $x_{\min} = \omega_n \sqrt{{1}/{s}}$.
Since $\omega_n$ can diverge arbitrarily slowly we see that the
extrinsic signal dimension $n$ plays no significant role in this bound,
and only the intrinsic dimension $s$ is relevant. Keep in mind,
however, that $\omega_n$ is related to the rate of convergence of the
risk $R(\hat\Phi_n)$ to zero. Corollary~\ref{coro:detection} is in
stark contrast to what is known for the same problem if one restricts
to the classical setting of non-adaptive sensing, as in Ingster \cite{ingster:97},
Ingster and Suslina \cite{ingster:03}, Donoho and Jin \cite{donoho:04}, Donoho \cite
{donoho:06}. For instance, for the
class of all subsets with cardinality $s$ it is necessary to have
$x_{\min} \geq c\sqrt{\log n}$ if $s=\mathrm{o}(\sqrt{n})$, where the factor
$c>0$ depends on the specific relation between $s$ and $n$. In
Meinshausen and Rice \cite{meinshausen:06}, the authors
considered estimation of the proportion
of significant components $|S|/n$. Their setting is more general, as
the distributions corresponding to significant and insignificant signal
component observations can be non-normal. Their approach can be used to
test the hypothesis $|S|=0$. For the Gaussian case, they recover
essentially the $\sqrt{\log n}$ scaling. Finally, in Cai, Jin and Low \cite{cai:07}
the authors consider again the estimation of the fraction of
significant signal components in the normal means case, and show
results beyond consistency, including minimax rates of convergence of
the risk. We now proceed with the proof of the theorem and a discussion
about tightness of the bounds.
\begin{pf*}{Proof of Theorem~\ref{thm:detection}}
The proof of this lower bound hinges, as usual, on the analysis of
likelihood ratios. Begin by defining the joint probability density
function of $\{Y_k,A_k,\Gamma_k\}_{k=1}^\infty$ under $S$, which we
denote by
\[
f(d;S)=f(y_1,a_1,\gamma_1,y_2,a_2,
\gamma_2,\ldots;S) .
\]
Note that this is properly defined for a certain dominating measure
(mixed continuous and discrete). Taking into account the conditional
dependences in our observation model we can factorize this probability
density function as follows
\begin{eqnarray*}
f(d;S) &=& f_{A_1,\Gamma_1}(a_1,\gamma_1) \times
f_{Y_1|A_1,\Gamma
_1}(y_1|a_1,\gamma_1;S)
\\
&&{} \times f_{A_2,\Gamma_2|Y_1,A_1,\Gamma_1}(a_2,\gamma _2|y_1,a_1,
\gamma_1)\times f_{Y_2|A_2,\Gamma_2}(y_2|a_2,
\gamma_2;S) \times\cdots.
\end{eqnarray*}
Note that in this factorization only some terms involve the underlying
true set $S$, while all the other terms depend solely on the sensing
strategy used. This greatly simplifies the computation of likelihood
ratios, as all the terms not involving $S$ cancel out. In particular,
the likelihood ratio between two hypotheses is given simply by
%
\begin{eqnarray}\label{eqn:LR}
\LR_{S,S'}(d) &=& \frac{f(d;S)}{f(d;S')}
\\
&=& \prod_{k=1}^\infty\frac{f_{Y_k|A_k,\Gamma_k}(y_k|a_k,\gamma
_k;S)}{f_{Y_k|A_k,\Gamma_k}(y_k|a_k,\gamma_k;S')}
.
\end{eqnarray}
As usual, in order to effectively distinguish if the underlying true
distribution is parameterized by $S$ or $S'$ the corresponding
likelihood ratio needs to be significantly different than 1. We proceed
by formally stating this. Our analysis is heavily inspired by the
approach in Chernoff \cite{chernoff:59}.

The first step is to relate the probabilities of type I and type II
errors to the likelihood ratio, namely giving a relation between $\P
_S(\hat\Phi\neq1)$ and $\P_\emptyset(\hat\Phi\neq\emptyset)$
where $S$ is an arbitrary element of $\C$.
Begin by defining the total variation and the Kullback--Leibler
divergence between two probability measures.
%
\begin{defn}
Let $\P_0$ and $\P_1$ be two probability measures defined on a common
measurable space $(\Omega,\mathcal{B})$. The total variation distance
is defined as
\[
\TV(\P_0,\P_1) = \sup_{B\in\mathcal{B}} \bigl|
\P_0(B)-\P_1(B)\bigr| .
\]
The Kullback--Leibler divergence is defined as
\[
\KL(\P_0\|\P_1)=\lleft\{ %
\begin{array} {l@{\qquad}l}
\displaystyle \int_\Omega\log\displaystyle \frac{\mathrm{d}\P_0}{\mathrm{d}\P_1} \,\mathrm{d}\P_0 & \mbox{if } \P_0 \ll\P_1 ,
\\
+\infty& \mbox{otherwise}. \end{array} %
\rright.
\]
\end{defn}

The total variation is a proper distance, unlike the Kullback--Leibler
divergence. Both are always non-negative but the latter is not
symmetric. If $f_0$ and $f_1$ are densities with respect to a measure
dominating both $\P_0$ and $\P_1$ the Kullback--Leibler divergence can
simply be written as
\[
\KL(\P_0\|\P_1)=\E_0 \biggl[\log
\frac{f_0(X)}{f_1(X)} \biggr] ,
\]
where $X$ is a random variable with distribution given by $\P_0$.
Therefore, this is the expected value of a log-likelihood ratio.
Consider now the setting in this paper. As done in Tsybakov \cite{tsybakov:09},
the total variation is closely related to the infimum of the sum of
type I and type II error probability, namely, for any binary (test)
function $\hat\Phi$ we have
\[
\P_\emptyset(\hat\Phi\neq0)+\P_S(\hat\Phi\neq1)\geq1-
\TV(\P _\emptyset,\P_S) .
\]
Evaluating the total variation distance is generally difficult, but
using Lemma~2.6 of Tsybakov \cite{tsybakov:09} we can
relate it to the
Kullback--Leibler divergence, which is generally much easier to
evaluate. Namely
\[
1-\TV(\P_\emptyset,\P_S)\geq\tfrac{1}{2}\exp\bigl(-\KL(
\P_\emptyset \|\P_S)\bigr) .
\]
Putting these two results together we obtain a simple relation between
the Kullback--Leibler divergence and the probabilities of error,
%
\begin{equation}\label{eqn:keylowerbound}
\KL(\P_\emptyset\|\P_S)\geq-\log \bigl(2\P_\emptyset(
\hat\Phi \neq\emptyset)+2\P_S(\hat\Phi\neq1) \bigr)
.
\end{equation}

To simplify the notation, let $\LR_{S,S'}\equiv\LR_{S,S'}(D)$. From
equation \eqref{eqn:keylowerbound} we conclude that
\[
\E_\emptyset[\log\LR_{\emptyset,S}]=\KL(\P_\emptyset\|
\P_S) \geq-\log \bigl(2\P_\emptyset(\hat\Phi\neq0)+2
\P_S(\hat\Phi \neq1) \bigr) .
\]
Since the choice of set $S$ was completely arbitrary, we have the bound
%
\begin{eqnarray}\label{eqn:logliklowerbound}
\min_{S\in\C} \E_\emptyset[\log\LR_{\emptyset,S}] \geq\min
_{S\in\C} \bigl\{-\log \bigl(2\P_\emptyset(\hat\Phi
\neq0)+2\P _S(\hat\Phi\neq1) \bigr) \bigr\} .
\end{eqnarray}

At this point it is important to note that, if we desire to have
$R(\hat\Phi)\leq\epsilon$ for some $0<\epsilon<1$ then $\P
_\emptyset(\hat\Phi\neq0)+\P_S(\hat\Phi\neq1)\leq\epsilon$
(for any $S\in\C$), and therefore
%
\begin{equation}\label{eqn:lower_LR_bound}
\min_{S\in\C} \E_\emptyset[\log\LR_{\emptyset,S}] \geq\log
\biggl(\frac{1}{2\epsilon} \biggr) .
\end{equation}
The next step of the proof entails deriving a good upper bound on $\min_{S\in\C} \E_\emptyset[\log\LR_{\emptyset,S}]$ and comparing it
to the lower bound just shown.

As noted before, the expected likelihood ratio is actually the
Kullback--Leibler divergence between $\P_\emptyset$ and $\P_S$. This
obviously depends on the sensing strategy $\mathcal{A}$ that is used.
Therefore, we need to get an upper bound on
%
\begin{equation}\label{eqn:maxmin}
\sup_{\mathcal{A}} \min_{S\in\C} \E_\emptyset[
\log\LR _{\emptyset,S}] .
\end{equation}
It is instructive to compare the above expression with the one of the
minimax error \eqref{eqn:minimax}. Note that the roles of the max/sup
and min/inf are reversed. This should not come as a surprise as larger
values of $\E_\emptyset[\log\LR_{\emptyset,S}]$ correspond to
lower probabilities of error. Note also that $\E_\emptyset[\log\LR
_{\emptyset,S}]$ can be interpreted as the payoff matrix of a game
where the sensing strategy makes the first move, and nature is the
opponent that chooses a sparsity pattern in an adversarial way. Now
note that
\begin{eqnarray*}
\E_\emptyset[\log\LR_{\emptyset,S}] &=& \sum
_{k=1}^\infty\E _\emptyset \biggl[\log
\frac{f_{Y_k|A_k,\Gamma_k}(Y_k|A_k,\Gamma
_k;\emptyset)}{f_{Y_k|A_k,\Gamma_k}(Y_k|A_k,\Gamma_k;S)} \biggr]
\\
&=& \sum_{k=1}^\infty\E_\emptyset\biggl[\E_\emptyset \biggl[\log\frac{f_{Y_k|A_k,\Gamma_k}(Y_k|A_k,\Gamma_k;\emptyset
)}{f_{Y_k|A_k,\Gamma_k}(Y_k|A_k,\Gamma_k;S)}\Bigl\vert
A_k,\Gamma _k \biggr] \biggr]
\\
&=& \sum_{k=1}^\infty\E_\emptyset
\biggl[\frac{\mu^2 \1\{A_k\in S\}
}{2} \Gamma^2_k \biggr]
\\
&=& \frac{\mu^2}{2} \sum_{k=1}^\infty
\E_\emptyset \bigl[\1\{ A_k\in S\} \Gamma^2_k
\bigr] ,
\end{eqnarray*}
where the final steps rely simply on the Kullback--Leibler divergence
between normal random variables with the same variance and different
means. At this point, we need to evaluate
\[
\sup_{\mathcal{A}} \min_{S\in\C} \Biggl\{
\frac{\mu^2}{2} \sum_{k=1}^\infty
\E_\emptyset \bigl[\1\{A_k\in S\} \Gamma^2_k
\bigr] \Biggr\} .
\]
We need to solve the above optimization problem over the space of all
possible sensing strategies. Although this might seem rather involved,
this optimization can be reduced to a much simpler deterministic
optimization problem. Begin by defining
%
\begin{equation}\label{eqn:b}
b_i=\sum_{k=1}^\infty
\E_\emptyset\bigl[\1\{A_k=i\} \Gamma^2_k
\bigr] .
\end{equation}
Note that this definition does not depend on $S$, as the expectation is
taken under the null hypothesis. Furthermore $b_i\geq0$, and the
sensing budget equation in the observation model \eqref{eqn:budget}
can be written as $\sum_{i=1}^n b_i\leq m$. Therefore,
\begin{eqnarray*}
&&\sup_{\mathcal{A}} \min_{S\in\C} \Biggl\{
\frac{\mu^2}{2} \sum_{k=1}^\infty
\E_\emptyset \bigl[\1\{A_k\in S\} \Gamma^2_k
\bigr] \Biggr\}
\\
& &\quad= \frac{\mu^2}{2} \sup_{\mathcal{A}} \min_{S\in\C}
\Biggl\{ \sum_{k=1}^\infty\sum
_{i\in S} \E_\emptyset \bigl[\1\{A_k=i\}
\Gamma^2_k \bigr] \Biggr\}
\\
&&\quad = \frac{\mu^2}{2} \sup_{\mathcal{A}} \min_{S\in\C}
\Biggl\{ \sum_{i\in S} \sum
_{k=1}^\infty\E_\emptyset \bigl[\1
\{A_k=i\} \Gamma^2_k \bigr] \Biggr\}
\\
&&\quad = \frac{\mu^2}{2} \sup_{\vec{b}\in\R_0^+: \sum_{i=1}^n b_i\leq
m} \min_{S\in\C}
\sum_{i\in S} b_i .
\end{eqnarray*}
We have now a relatively simple finite dimensional problem, where we
seek to identify the vector $\vec{b}=(b_1,\ldots,b_n)$) maximizing a
concave function. The solution of this problem obviously depends on the
exact structure of $\C$. Remarkably, for symmetric classes, the
solution is extremely simple and characterized in the first part of the
following lemma, proved in the Appendix.
%
\begin{lemma}\label{lemma:optimization}
Let $\C$ be a symmetric class. Let $\Xi=\bigcup_{S\in\C} S$. Then
\begin{enumerate}[2.]
\item[1.]
\[
\sup_{\vec{b}\in\R_0^+: \sum_{i=1}^n b_i=m} \min_{S\in\C} \sum
_{i\in S} b_i =\frac{ms}{|\Xi|} ,
\]
\item[2.]
\[
\sup_{\vec{b}\in\R_0^+: \sum_{i=1}^n b_i=m} \frac{1}{|\C|} \sum
_{S\in\C} \sum_{i\in S} b_i
=\frac{ms}{|\Xi|} ,
\]
\end{enumerate}
and in both cases the solution is attained taking $b_i=m/|\Xi|$ for
$i\in\Xi$ and zero otherwise.
\end{lemma}

We are now in place to prove the theorem: by putting together the
likelihood ratio lower bound \eqref{eqn:lower_LR_bound} and the above
upper bound we get
\[
\frac{\mu^2 m s}{2|\Xi|}\geq\log\frac{1}{2\epsilon} ,
\]
which is equivalent to
\[
\mu\geq\sqrt{\frac{2 |\Xi|}{s m} \log\frac{1}{2\epsilon}}
\]
concluding the proof.
\end{pf*}

Lower bounds for adaptive sensing in settings other than the one in
this paper have been derived previously. For instance, in Castro and Nowak \cite{castro:08} a minimax characterization of the
fundamental performance
limits of active learning for a binary classification problem was
provided. Such results were made possible by bringing together
approximation results for smooth functional spaces and classical
minimax bounding techniques (as in Tsybakov \cite
{tsybakov:09}), modified to
incorporate the sequential experimental design aspect of the problem.
In that approach the functional approximation results played the
prominent role, and the stochastic part of the error had a much smaller
contribution. Unfortunately this is not the case for the setting
considered in the current paper and previously existing approaches were
not adequate, prompting the novel approach presented here.

The proof of this theorem can be adapted for the other two risk
definitions \eqref{eqn:max_risk} and \eqref{eqn:bayes_risk}, and we
can show that the risk behavior is qualitatively the same. These
results are stated in the following proposition, proved in the \hyperref[app]{Appendix}.
%
\begin{prop}\label{prop:detection_other_risks}
Consider the setting of Theorem~\ref{thm:detection} and let
$0<\epsilon<1$. If $\tilde R(\hat\Phi)\leq\epsilon/2$\vspace*{2pt} or $\bar
R(\hat\Phi)\leq\epsilon$ then the conclusion of Theorem~\ref
{thm:detection} is still valid and the lower bound \eqref
{eqn:detection} holds.
\end{prop}

\subsection{Tightness of the detection lower bounds}\label
{sec:tightness_detection}

We now proceed to show that the lower bounds derived above are indeed
tight, in the sense that there are adaptive sensing testing procedures
which are able to nearly attain them. As we saw, for symmetric classes
$\C$, extra class structure does not help. Therefore, we focus
exclusively on the largest class of all the subsets of $\{1,\ldots,n\}
$ with cardinality $s$. In Haupt, Castro and Nowak \cite{haupt:10}, a
procedure called
Distilled Sensing (DS) was introduced, and the authors proved that for
the detection problem described above this procedure is able to
asymptotically drive the risk to zero when $\mu>4\sqrt{n/m}$ and
$\log\log\log n<s<n^{1-\beta}$ for some $\beta\in(0,1)$. When
comparing this result to the above lower bound we see that there is a
huge gap, as we would expect the signal magnitude $\mu$ to scale
essentially like $\sqrt{2n/(sm)}$. However, it is important to note
that DS is entirely agnostic about the sparsity level and possible
signal magnitude. An alternative non-agnostic methodology can be
derived using DS as a black-box, which nearly achieves the lower-bounds
of the previous section.

We begin by formally stating the performance results for the DS
procedure. The following proposition is essentially the second part of
Theorem III.1 in Haupt, Castro and Nowak \cite{haupt:10}.
%
\begin{prop}[(From Haupt, Castro and Nowak \cite{haupt:10}\footnote{The
sparsity lower bound condition $\log\log\log n < s$ is not stated in the theorem in
Haupt, Castro and Nowak \cite{haupt:10} for presentation reasons, and
the discussion on the validity of the result for $\log\log\log n < s$ appears only on the last
paragraph of Section VI.})]
\label{prop:DS}
Assume $\log\log\log n< s\leq n^{1-\beta}$, for some $\beta\in
(0,1)$.\vspace*{2pt} Furthermore let $\mu>4\sqrt{n/m}$. There is a sensing
strategy $\mathcal{A}_\DS$ and a test function $\hat\Phi_\DS$ such that
\[
R(\hat\Phi_\DS)\rightarrow0 ,
\]
as $n\rightarrow\infty$.
\end{prop}

Note that this result is valid even if $s\approx\log\log\log n$,
meaning $s$ is nearly asymptotically constant. This suggests the
following modification: first randomly select $\tilde n$ elements of $\{
1,\ldots,n\}$ without replacement. Denote these by $\mathcal{E}=\{
E_1,\ldots,E_{\tilde n}\}$. Our sensing strategy will focus
exclusively on the entries $\mathcal{E}$ and ignore all the remaining
ones. In other words, our observation model is now
\[
Y_k=x_{E_{A_k}}+\Gamma^{-1}_k
W_k \qquad \forall k\in\{1,2,\ldots\} ,
\]
where $A_k\in\{1,\ldots,\tilde n\}$. The sensing budget is, however,
the same as in the original formulation
\[
\sum_{k=1}^\infty\Gamma^2_k
\leq m .
\]
In summary, we have exactly the same setting as before, but the
extrinsic dimension $n$ is now replaced by the smaller $\tilde n$. Now,
provided we choose $\tilde n$ large enough so that the conditions of
Proposition~\ref{prop:DS} are met for this new setting then an
improvement in performance is possible, yielding the following result.
%
\begin{prop}\label{prop:MDS}
Assume $s>\log\log\log n$. Furthermore, let $\mu>\sqrt{\frac
{32n\log\log\log n}{sm}}$. There is an adaptive sensing testing
strategy such that
\[
R(\hat\Phi)\rightarrow0 ,
\]
as $n\rightarrow\infty$.
\end{prop}

This result means that the statement of Corollary~\ref{coro:detection}
is essentially tight, at least provided there are more than $\log\log
\log n$ signal components under the alternative hypothesis. The
constant in the bound is certainly not optimal, and the factor $\log
\log\log n$ is (possibly) an artifact of the procedure. Closing the
small gap between the upper and lower bounds is, however, still a
direction for future research.
%
\begin{remark}\label{rmk:tightness}
The results above were derived assuming the non-zero signal components
are positive. Qualitatively these results remain the same even if one
allows both positive and negative components. A simple way to address
this setting is to write $\vec{x}$ as $\vec{x}=\vec{x}^{+}-\vec
{x}^{-}$, where $\vec{x}^{+}$ and $\vec{x}^{-}$ are sparse signal
vectors with positive components (and the joint number of non-zero
components is simply $s$). Now we can split the sensing budget into two
equal parts, and make use of each one to test for the presence/absence
or either signal. This approach yields the same asymptotic behavior,
and will at most result in larger constants in the bounds.

Also note that, in principle, a procedure in the spirit of the one
introduced in Chernoff \cite{chernoff:59} could be
used to construct an
adaptive sensing and testing methodology. However, the method of
analysis in that paper is not adequate to deal with our setting.
Nevertheless such procedure seems to work extremely well based on a
short simulation study we conducted, and its analytical
characterization presents an interesting direction for future work.
\end{remark}
\begin{pf*}{Proof of Proposition~\ref{prop:MDS}}
The idea is simply to use the construction above, with $\tilde n=\frac
{2 n\log\log\log n}{s}$. Because of the random entry selection step
(the choice of $\mathcal{E}$) the conditions of Proposition~\ref
{prop:DS} might not always be satisfied. However, this happens with
very low probability. Define $\tilde x\in\R^{\tilde n}$ where $\tilde
x_i=x(E_i), i=1,\ldots,\tilde n$. Suppose $x$ has $s$ non-zero
components, and let $\tilde s$ be the number of non-zero components of
$\tilde x$. Because of the sampling without replacement process,
$\tilde s$ is an hypergeometric random variable with mean
\[
\E[\tilde s]=\tilde n \frac{s}{n} = 2\log\log\log n ,
\]
and variance
\[
\V(\tilde s)=\tilde n \frac{s}{n} \biggl(1-\frac{s}{n} \biggr)
\frac
{n-\tilde n}{n-1}\leq\tilde n\frac{s}{n}=2\log\log\log n .
\]
This means that
\begin{eqnarray*}
\P(\tilde s < \log\log\log n) &=& \P\bigl(\tilde s -\E[\tilde s]< \log \log\log
n-\E[\tilde s]\bigr)
\\
&=& \P\bigl(\tilde s -\E[\tilde s]< -\log\log\log n\bigr)
\\
&\leq& \P\bigl(\bigl|\tilde s -\E[\tilde s]\bigr|> \log\log\log n\bigr)
\\
&\leq& \frac{\V(\tilde s)}{(\log\log\log n)^2}
\\
&\leq& \frac{2}{\log\log\log n},
\end{eqnarray*}
where we used Chebyshev's inequality on the second-to-last step. This
means that, with probability at least $1-2/\log\log\log n$ the
conditions of Proposition~\ref{prop:DS} are fulfilled. For
convenience, define the event $\Omega=\{\tilde s \geq\log\log\log
n\}$. Since the detection risk is always bounded by 2, we have
\[
R(\hat\Phi)\leq2 \frac{2}{\log\log\log n} + R(\hat\Phi|\Omega) ,
\]
therefore it suffices to show that, conditionally on $\Omega$, the
risk of our procedure vanishes asymptotically. From Proposition~\ref
{prop:DS}, we know that if $\mu>4\sqrt{\tilde n/m}$ the detection
risk converges to zero, which immediately yields
\[
\mu>4\sqrt{\frac{2n\log\log\log n}{sm}},
\]
concluding the proof.
\end{pf*}

\section{Signal estimation}

In this section we consider the signal estimation problem, where the
goal is to identify the support $S$ of the underlying signal $\vec{x}$
as accurately as possible. As in the detection case, we are interested
in characterizing the minimum signal amplitude $x_{\min}$ for which
estimation is still possible. Clearly estimation is statistically more
``difficult'' than signal detection, and therefore the requirements on
$x_{\min}$ are more stringent in this case. Nevertheless we show that
the dependence on the extrinsic dimension $n$ does not play a
significant role in the asymptotic performance bounds.

For the same reasons as in the previous section, we focus our attention
on the signal model in \eqref{eqn:simple_model}. Our main goal is the
estimation of the signal support set $S=\{i:x_i\neq0\}$. In other
words, our goal is to use adaptive sensing observations to construct an
estimate $\hat S$ which is ``close'' to $S$. The metric of interest is
the cardinality of the symmetric set difference
\[
d(\hat S,S)=|\hat S\Delta S|=\bigl|\bigl(\hat S\cap S^c\bigr)\cup\bigl(
\hat S^c\cap S\bigr)\bigr| ,
\]
where $S^c$ denotes the complement of $S$ in $\{1,\ldots,n\}$. Clearly
$d(\hat S,S)$ is just the number of errors in the estimate $\hat S$. In
a similar spirit to that of the previous section, we want to determine
how small can the signal magnitude $\mu$ be so that
%
\begin{equation}\label{eqn:risk_numerrors}
\max_{S\in\C} \E_S\bigl[d(\hat S,S)\bigr]\leq
\epsilon,
\end{equation}
where $\C$ is a class of sets, and $\epsilon>0$ is small. A different
error metric which is also popular in the literature is $\P_S(\hat
S\neq S)$, that is, the probability one does not achieve exact support
estimation. Clearly
\[
\P_S(\hat S\neq S)\leq\E_S\bigl[d(\hat S,S)\bigr] ,
\]
and therefore this is a less stringent metric. The tools developed in
this paper pertain $\E_S[d(\hat S,S)]$ and it is not clear if adaptive
sensing lower bounds for $\P_S(\hat S\neq S)$ can be derived easily
using a similar approach.

In addition, we will also consider a different support estimation risk
function. Define the \emph{False Discovery Rate} (FDR) and the \emph
{Non-Discovery Rate} (NDR) as
\[
\FDR(\hat S,S)=\E_S \biggl[\frac{|\hat S \setminus S|}{|\hat
S|} \biggr]\quad  \mbox{and}\quad
\NDR(\hat S,S)=\E_S \biggl[\frac{|S
\setminus\hat S|}{|S|} \biggr] .
\]
In the above definitions convention $0/0=0$. Ideally we want both these
quantities to be as small as possible, and so we can naturally define
the risk
\[
R_{\FDR+\NDR}(\hat S,S)=\max_{S\in\C} \bigl\{\FDR(\hat S,S)+
\NDR (\hat S,S) \bigr\} .
\]
Obviously $\E_S[d(\hat S,S)]\geq\FDR(\hat S,S)+\NDR(\hat S,S)$ and
these two measures of error can be dramatically different, therefore
controlling the risk $R_{\FDR+\NDR}(\hat S,S)$ is significantly
easier than controlling the absolute number of errors.

Our original goal is to study lower bounds for the class $\C$ of all
subsets of $\{1,\ldots,n\}$ with cardinality $s$. For technical
reasons this is a bit challenging, and to greatly simplify the analysis
we consider a different setting that nonetheless captures the essence
of the problem. Let $\C'$ denote the class consisting of sets of
cardinality $s$, $s+1$ and $s-1$. This class is only ``slightly''
bigger than $\C$. We instead consider procedures that exhibit good
performance when $S\in\C'$, that is, estimation procedures that are
``very mildly'' adaptive to unknown sparsity. Generalization of the
results to other classes of sets shall be considered in future work and
is out of the scope of this paper.

To aid in the presentation, we introduce some new notation. Namely let
$S_i=\1\{i\in S\}$. Similarly, for any estimator $\hat S$ let $\hat
S_i=\1\{i\in\hat S\}$. Note that the joint description of $\hat S_i$
for all $i$ is equivalent to $\hat S$. For analysis purposes, it is
convenient to consider only \emph{symmetric} procedures, meaning that
for any $S\in\C'$
%
\begin{equation}\label
{eqn:symmetry1}
\forall i,j\in S \qquad \P_S(\hat S_i \neq1)=
\P_S(\hat S_j \neq1)
\end{equation}
and
%
\begin{equation}\label{eqn:symmetry2}
\forall i,j\notin S\qquad  \P_S(\hat S_i \neq0)=
\P_S(\hat S_j \neq0) .
\end{equation}
{\sloppy Although this might seem overly restrictive, it is indeed not
the case. Any inference procedure can be ``symmetrized'' without
increasing its maximal risk. In other words, given an estimator $\hat
S$ we can construct another estimator $\hat S^{(\perm)}$ satisfying
\eqref{eqn:symmetry1} and \eqref{eqn:symmetry2} and such that
\[
\E_S\bigl[d\bigl(\hat S^{(\perm)},S\bigr)\bigr] \leq\max
_{S'\in\C'} \E_{S'}\bigl[d\bigl(\hat S,S'
\bigr)\bigr] ,
\]
for all sets $S\in\C'$. The symmetrization is achieved by
randomization. Let ${\perm\dvtx \{1,\ldots,n\}\rightarrow\{1,\ldots,n\}
}$ be a permutation of $\{1,\ldots,n\}$ chosen uniformly at random
among the set of $n!$ possible permutations. Let $\hat S$ be a
particular estimator we are going to symmetrize. Proceed by exchanging
the identity of the entries of $\vec{x}$ using this permutation, or
equivalently by taking $A^{(\perm)}_k=A_{\perm^{-1}(k)}$ for all $k$,
and use the estimator $\hat S$ on the collected data. Finally, reverse
the permutation, namely defining $\hat S^{(\perm)}_i=\hat S_{\perm
(i)}$, for all $i\in\{1,\ldots,n\}$. Using this construction, we get
the following lemma, proved in the \hyperref[app]{Appendix}.
%
\begin{lemma}\label{lemma:symmetrization}
Let $\hat S$ be any adaptive sensing procedure. The random
symmetrization approach described in the paragraph above yields another
adaptive sensing procedure $\hat S^{(\perm)}$ such that, for any $S\in
\C'$
\[
\forall i\in S \qquad \P\bigl(\hat S^{(\perm)}_i\neq1\bigr)=
\frac{1}{|S|{n \choose
|S|}}\sum_{S'\in\C':|S'|=|S|} \sum
_{j\in S'} \P_{S'}(\hat S_j \neq 1)
\]
and
\[
\forall i\notin S\qquad  \P\bigl(\hat S^{(\perm)}_i\neq0\bigr)=
\frac{1}{(n-|S|){n
\choose|S|}} \sum_{S'\in\C:|S'|=|S|} \sum
_{j\notin S'} \P _{S'}(\hat S_j \neq0) .
\]
In addition, the following is also true:\vspace*{1pt}
\[
\E_S\bigl[d\bigl(\hat S^{(\perm)},S\bigr)\bigr]\leq
\frac{1}{{n \choose|S|}} \sum_{S'\in\C:|S'|=|S|} \E_{S'}\bigl[d
\bigl(\hat S,S'\bigr)\bigr] \leq\max_{S'\in\C
':|S'|=|S|}
\E_{S'}\bigl[d\bigl(\hat S,S'\bigr)\bigr] .
\]
\end{lemma}
}
This ensures that without loss of generality we can consider only
symmetric procedures. It is important to note that this approach
is valid only if the class $\C'$ is invariant under permutations.
Finally, for symmetric procedures the lower bounds we derive are
also applicable to measures of risk different than \eqref
{eqn:risk_numerrors}, such as the \emph{average estimation risk}
$\frac{1}{|\C'|}\sum_{S'\in\C'} \E_{S'}[d(\hat S,S')]$.

\subsection{Main results -- Estimation}
%
\begin{theorem}\label{thm:estimation}
Let $\C'$ denote the class of all subsets of $\{1,\ldots,n\}$ with
cardinality $s$, $s+1$ and $s-1$. Let $\hat S\equiv\hat S(D)$ be an
arbitrary adaptive sensing estimator, where $D=\{Y_i,A_i,\Gamma_i\}
_{i=1}^{\infty}$. If
\[
\max_{S\in\C'} \E_S\bigl[d(\hat S,S)\bigr]\leq
\epsilon,
\]
where $0<\epsilon<1$ then necessarily
\[
x_{\min} \geq\sqrt{\frac{2n}{m} \biggl(\log s + \log
\frac
{n-s}{n+1} + \log\frac{1}{2\epsilon} \biggr)} .
\]
\end{theorem}

The proof of the theorem is presented at the end of this section. As
before it is useful to look at the asymptotic behavior, and the case
$s\ll n$ is particularly interesting.
%
\begin{coro}\label{coro:estimation}
Consider the setting of Theorem~\ref{thm:estimation} and assume
$s=\mathrm{o}(n)$ as $n\rightarrow\infty$. Let $\hat S_n$ be an arbitrary
estimation procedure for which
\[
\lim_{n\rightarrow\infty} \max_{S\in\C'} \E_S
\bigl[d(\hat S_n,S)\bigr]=0 .
\]
Necessarily
\[
x_{\min} \geq\sqrt{2\frac{n}{m} (\log s+\omega_n )} ,
\]
where $\omega_n$ is a sequence for which $\lim_{n\rightarrow\infty}
\omega_n=\infty$.
\end{coro}

For the $\FDR+\NDR$ risk, we can use the same proof approach to obtain a
much less restrictive bound on the signal magnitude.
%
\begin{coro}\label{coro:estimation_FDR}
Consider the setting of Theorem~\ref{thm:estimation} and assume
$s=\mathrm{o}(n)$. Let $\hat S_n$ be an arbitrary estimation procedure such that
\[
\lim_{n\rightarrow\infty} R_{\FDR+\NDR}(\hat S,S)=0 .
\]
Necessarily
\[
x_{\min} \geq\omega_n\sqrt{\frac{n}{m}} ,
\]
where $\omega_n$ is a sequence for which $\lim_{n\rightarrow\infty}
\omega_n=\infty$.
\end{coro}

A sketch of the proof of this corollary can be found in the \hyperref[app]{Appendix}.
\begin{pf*}{Proof of Theorem~\ref{thm:estimation}}
The proof follows a similar approach to that of Theorem~\ref
{thm:detection}, and capitalizes heavily on the symmetry of the
estimation procedure. In light of Lemma~\ref{lemma:symmetrization}, it
suffices to consider symmetric procedures, that is, procedures that
satisfy \eqref{eqn:symmetry1} and \eqref{eqn:symmetry2}. Let $S\in\C
'$ be arbitrary and assume that
\[
\E_S\bigl[d(\hat S,S)\bigr]\leq\epsilon,
\]
where $0<\epsilon<1$. Clearly
\begin{eqnarray*}
\E_S\bigl[d(\hat S,S)\bigr] &=& \E_S \Biggl[ \sum
_{i=1}^n \1\{\hat S_i \neq
S_i\} \Biggr]
\\
&=& \sum_{i\in S} \E_S \bigl[\1\{\hat
S_i \neq1\} \bigr] + \sum_{j\notin S}
\E_S \bigl[\1\{\hat S_j \neq0\} \bigr]
\\
&=& \sum_{i\in S} \P_S (\hat
S_i \neq1 ) + \sum_{j\notin
S}
\P_S (\hat S_j \neq0 ) .
\end{eqnarray*}
As we consider symmetric procedures, we conclude that
\[
\forall i\in S \qquad \P_S(\hat S_i \neq1)\leq
\frac{\epsilon}{|S|}
\]
and
\[
\forall i\notin S\qquad  \P_S(\hat S_i \neq0)\leq
\frac{\epsilon}{n-|S|} .
\]

For our purposes, it is convenient to rewrite the likelihood ratio
\eqref{eqn:LR} as
\begin{eqnarray*}
\LR_{S,S'}(d) &=& \frac{f(d;S)}{f(d;S')}
\\
&=& \prod_{i=1}^n \prod
_{k:a_k=i} \frac{f_{Y_k|A_k,\Gamma
_k}(y_k|a_k,\gamma_k;S)}{f_{Y_k|A_k,\Gamma_k}(y_k|a_k,\gamma_k;S')} .
\end{eqnarray*}

Now let $S\in\C$ be an arbitrary set of cardinality $s$, and define
$S^{(i)}\in\C'$ to be
\[
S^{(i)}= \left\{ %
\begin{array} {l@{\qquad}l} S\setminus\{i\} & \mbox{if
} i\in S,
\\
S\cup\{i\} & \mbox{if } i\notin S, \end{array} %
\rright.
\]
in words, we either remove element $i$ if $i\in S$, or add it
otherwise, meaning that $S\Delta S^{(i)}=\{i\}$. We proceed in a similar
way as we did in the signal detection scenario.\vadjust{\goodbreak} Let ${i\in\{ 1,\ldots,n\}}$ be
arbitrary. We conclude that %
\[
\forall i\in S\qquad  \E_S [
\log\LR_{S,S^{(i)}} ] \geq-\log \bigl(2\P_S (\hat S_i
\neq1 )+2\P_{S^{(i)}} (\hat S_i \neq0 ) \bigr)
\]
and
\[
\forall i\notin S \qquad \E_S [ \log\LR_{S,S^{(i)}} ] \geq -\log
\bigl(2\P_S (\hat S_i \neq0 )+2\P_{S^{(i)}} (\hat
S_i \neq1 ) \bigr) .
\]
We now take advantage of the symmetry of the estimator, to conclude that\vspace*{-1pt}
%
\begin{equation}\label
{eqn:LR_lowerbound1}
\forall i\in S\qquad  \E_S [ \log\LR_{S,S^{(i)}} ] \geq-\log \biggl(
\frac{2\epsilon}{s}+\frac{2\epsilon}{n-s+1} \biggr)
\end{equation}
and\vspace*{-1pt}
%
\begin{equation}\label{eqn:LR_lowerbound2}
\forall i\notin S \qquad \E_S [ \log\LR_{S,S^{(i)}} ] \geq -\log
\biggl(\frac{2\epsilon}{n-s}+\frac{2\epsilon}{s+1} \biggr) .
\end{equation}

Now that we have lower bounds for $\E_S  [ \log\LR
_{S,S^{(i)}} ]$ we need to evaluate this quantity in terms of
$\mu$. This is easily done by noting that for $i\in\{1,\ldots,n\}$\vspace*{-1pt}
\begin{eqnarray*}
\E_S [ \log\LR_{S,S^{(i)}} ] &=& \E_S \biggl[\sum
_{k:A_k=i} \log\frac{f_{Y_k|A_k,\Gamma_k}(Y_k|A_k,\Gamma
_k;S)}{f_{Y_k|A_k,\Gamma_k}(Y_k|A_k,\Gamma_k;S^{(i)})} \biggr]
\\
&=&
\E_S \biggl[\sum_{k:A_k=i} \frac{\mu^2}{2}
\Gamma^2_k \biggr] .
\end{eqnarray*}
Note that we cannot yet evaluate the above expression, as one cannot
invoke the sensing budget constraint \eqref{eqn:budget}. This can be
addressed by summing each of the above terms over $i\in\{1,\ldots,n\}
$. On one hand\vspace*{-1pt}
%
\begin{equation}\label{eqn:upper_bound_estimation}
\sum_{i=1}^n \E_S [ \log
\LR_{S,S^{(i)}} ] = \E_S \Biggl[\sum_{i=1}^n
\sum_{k:A_k=i} \frac{\mu^2}{2}\Gamma^2_k
\Biggr] = \E_S \Biggl[\sum_{k=1}^\infty
\frac{\mu^2}{2}\Gamma^2_k \Biggr] \leq
\frac{m \mu^2}{2} .
\end{equation}
On the other hand\vspace*{-1pt}
\begin{eqnarray*}
\sum_{i=1}^n \E_S [ \log
\LR_{S,S^{(i)}} ] &=& \sum_{i\in S} \E_S
[ \log\LR_{S,S^{(i)}} ]+\sum_{i\notin S}
\E_S [ \log\LR_{S,S^{(i)}} ]
\\
&\geq& -s \log \biggl(\frac{2\epsilon}{s}+\frac{2\epsilon
}{n-s+1} \biggr)-(n-s)\log
\biggl(\frac{2\epsilon}{n-s}+\frac
{2\epsilon}{s+1} \biggr) .
\end{eqnarray*}
We can get a more insightful bound by reorganizing the various terms\vspace*{-1pt}
\begin{eqnarray*}\label{eqn:above}
&&\sum_{i=1}^n \E_S [ \log
\LR_{S,S^{(i)}} ] \nonumber\\
&&\quad \geq n\log \frac{1}{2\epsilon}+s\log\frac{s(n-s+1)}{n+1}+(n-s)
\log\frac
{(n-s)(s+1)}{n+1}
\nonumber
\\
&&\quad = n\log\frac{1}{2\epsilon}+s\log s +(n-s)\log(s+1)+s\log\frac
{n-s+1}{n+1}+(n-s)
\log\frac{n-s}{n+1}
\nonumber\\
&&\quad \geq n \biggl(\log s+\log\frac{n-s}{n+1}+\log\frac{1}{2\epsilon
} \biggr),
\nonumber
\end{eqnarray*}
where the last inequality follows by noting that $\log(s+1)>\log s$
and $\log(n-s+1)>\log(n-s)$. Using this together with \eqref
{eqn:upper_bound_estimation} concludes the proof.
\end{pf*}

\subsection{Tightness of the estimation lower bounds}\label
{sec:tightness_estimation}

Similarly to what happened in the detection setting the lower bounds
derived for estimation are also tight, in the sense that there are
inference procedures able to achieve them. In Malloy and Nowak \cite{malloy:11,malloy:11b}, a slightly
different problem was considered,
where each measurement had the same accuracy/precision and one desired
to control the total number of errors in $\hat S$. Their results were
stated in term of conditions on the signal magnitude $\mu$ that were
necessary to ensure the risk converged to zero. In their setting, there
is no strict sensing budget, but instead only control over the expect
precision budget used. In other words, the procedures in Malloy and Nowak \cite{malloy:11,malloy:11b} do not always
satisfy the sensing budget in
equation \eqref{eqn:budget}, but instead satisfy an \emph{expected}
sensing budget constraint
\[
\E \Biggl[\sum_{k=1}^\infty
\Gamma^2_k \Biggr] \leq m .
\]
Such methods can be modified to ensure that the sensing budget \eqref
{eqn:budget} is fulfilled with increasingly high probability (as $n$
grows) without altering their asymptotic performance behavior, and we
can state the following result, proved in the \hyperref[app]{Appendix}.
%
\begin{prop}\label{prop:estimation_upperbound}
Assume $s \leq\frac{n}{(\log_2^2 n) - 3}$. Let
\[
\mu\geq\sqrt{\frac{4n}{m}(2\log s+5\log\log_2 n)} .
\]
There is a sensing and estimation strategy yielding an estimator $\hat
S$ such that
\[
\max_{S\in\C'} \E_S\bigl[d(\hat S,S)\bigr]
\rightarrow0 ,
\]
as $n\rightarrow\infty$.
\end{prop}
This means that provided $x_{\min}$ is of the order $\sqrt{(n/m)(\log
s +\log\log n)}$ we can ensure exact recovery of a sufficiently sparse
signal support with probability approaching 1. The\vadjust{\goodbreak} proposition is
proved in the \hyperref[app]{Appendix}. The constants in the above result are rather
loose, and can be made much tighter (see Malloy and Nowak \cite
{malloy:11}). The $\log
\log n$ term is an artifact of this method (which is parameter adaptive
and agnostic about $s$). This term can be entirely avoided by
considering another procedure, namely by executing in parallel $n$
properly calibrated sequential likelihood ratio tests, which requires
the knowledge of the sparsity level $s$. Such a procedure achieves
precisely the bound in Corollary~\ref{coro:estimation}. Lower bounds
for estimation have been derived under a different set of assumptions
for the class of entry-wise sequential tests in Malloy and Nowak \cite{malloy:11b}. In
contrast, the results in the current paper pertain any adaptive sensing
procedure (and not only entry-wise testing procedures).

Control of the $\FDR+\NDR$ risk was considered in Haupt, Castro and Nowak
\cite{haupt:10} in the
exact setting described in this paper, and the distilled sensing
procedure proposed there is able to achieve the bound in Corollary~\ref
{coro:estimation_FDR} provided $\log\log\log n < s \leq n^{1-\beta}$
for some $0<\beta<1$. Therefore the lower bounds on the $\FDR+\NDR$ risk
are also tight for a wide range of sparsity levels.

\subsection{Relation to compressed sensing}

The proof technique used in Theorem~\ref{thm:estimation} provides some
important insights for the problem of adaptive compressive sensing.
This setting is different than the one considered so far and the
observation model is now of the form
\[
\vec{Y}=\vec{A}\vec{x}+\vec{W} ,
\]
where $\vec{Y}\in\R^l$ denotes the observations, $\vec{A}\in\R
^{l\times n}$ is the design/sensing matrix, $\vec{x}\in\R^n$ is the
unknown signal, and $\vec{W}\in\R^l$ is Gaussian with zero mean an
identity covariance matrix. The rows of $\vec{A}$ can be designed
sequentially, and the $k$th row (denoted by $\vec{A}_{k\cdot}$) can
depend explicitly on $\{Y_j,\vec{A}_{j\cdot}\}_{j=1}^{k-1}$. Note
that $W_k$ is a normal random variable independent of $\{Y_j,\vec
{A}_{j\cdot},W_j\}_{j=1}^{k-1}$ and also independent of $\vec
{A}_{k\cdot}$. This setting is particularly interesting when we impose
some constraints on $\vec{A}$, namely
\[
\E \bigl[\|\vec{A}\|_F^2 \bigr] \leq m ,
\]
where $\|\cdot\|_F$ is the Frobenius matrix norm. Like \eqref
{eqn:budget}, this sensing budget condition is very natural and the
issue of noise is irrelevant without it. Each row $\vec{A}_{k\cdot}$
plays the role of the sensing action $A_k$ in our original scenario,
and $\|\vec{A}_{k\cdot}\|_2^2$ plays the role of the precision
parameter $\Gamma^2_k$ in \eqref{eqn:budget}. As before, we do not
impose any restrictions on the total number of measurements $l$, which
can be potentially infinite. We can show the following result using an
approach similar to that of Theorem~\ref{thm:estimation}.
%
\begin{prop}\label{prop:CS_lowerbound}
Consider the adaptive compressed sensing setting as described above,
with observations $\vec{Y}=\vec{A}\vec{x}+\vec{W}$, where $\vec
{W}$ is Gaussian zero mean with identity covariance matrix and $\E
[\|\vec{A}\|_F^2 ] \leq m$. Let $\mathcal{H}(\mu)\subset
\R^n$ be the class of all vectors $\vec{x}$ with support in $\C'$
(i.e. the support\footnote{Define $\supp(\vec{x})=\{i:x_i\neq0\}$.}
has cardinality $s$, $s+1$ or $s-1$) and the magnitude of the minimum
non-zero entries greater or equal than $\mu$. That is
\[
\mathcal{H}(\mu)=\Bigl\{\vec{x}\in\R^n\dvtx  \supp(x)\in\C'
\mbox{ and } \min_i\bigl\{|x_i|\dvtx x_i
\neq0\bigr\}\geq\mu\Bigr\} .
\]
Let $D=\{\vec{Y},\vec{A}\}$ and $\hat S(D)$ be an arbitrary
estimator. If
%
\begin{equation}\label{eqn:CS_condition}
\max_{\vec{x}\in\mathcal{H}{\mu}} \E_{\vec{x}}\bigl[d(\hat S,S)\bigr]\leq
\epsilon,
\end{equation}
where $0<\epsilon<1$ then necessarily
\[
\mu\geq\sqrt{\frac{2n}{m} \biggl(\log s + \log\frac{n-s}{n+1} + \log
\frac{1}{2\epsilon} \biggr)} .
\]
\end{prop}

The proof of the proposition can be found in the \hyperref[app]{Appendix}. In
Arias-Castro, Cand\`{e}s and Davenport \cite{arias-castro:11}, the authors
derived lower bounds for both support
recovery and mean square error risk for adaptive compressive sensing.
In their setting $l=m$, and each row of the matrix $\vec{A}$ has
expected norm at most 1. These two constraints imply the Frobenius norm
constraint in Proposition~\ref{prop:CS_lowerbound}. Theorem~2 in that
paper states that the minimum signal amplitude $x_{\min}$ must be
greater than $\sqrt{n/m}$ to ensure that support recovery is possible
within the class of all possible $s$-sparse signals. In contrast, our
result shows that the lower bound is not entirely tight. Formally, if
$s=\mathrm{o}(n)$ and
\[
\lim_{n\rightarrow\infty} \max_{S\in\C'} \E_S
\bigl[d(\hat S_n,S)\bigr]=0
\]
we have necessarily
\[
x_{\min} = \sqrt{2\frac{n}{m} (\log s +\omega_n )} ,
\]
as $n\rightarrow\infty$. So, the above result improves the bound in
Arias-Castro, Cand\`{e}s and Davenport \cite{arias-castro:11} by a $\log s$
factor. In light of the recent
results in Haupt \textit{et al.} \cite{haupt:12}, it seems plausible
that this is a
necessary and sufficient term. However, a precise characterization of
these limits remains an open problem.

\section{Conclusion}

In this paper, we presented several lower bounds for detection and
estimation of sparse signals using adaptive sensing. These results
bridge a gap in our understanding of adaptive sensing and show that
methodologies recently proposed in the literature are nearly optimal. A
very interesting insight is that, for signal detection, the sparsity
structure is essentially irrelevant. The intuition being that for
detection it suffices to identify one non-zero component, and cues
provided\vadjust{\goodbreak} by the structure are not too useful under adaptive sensing
scenarios. However, for signal estimation it is not clear if structure
helps, which raises many interesting directions for future research.

\begin{appendix}

\section*{Appendix}\label{app}
\renewcommand{\theequation}{5.\arabic{equation}}
\setcounter{equation}{0}
\begin{pf*}{Proof of Lemma~\ref{lemma:optimization}}
We begin by proving the first result. Let
\[
b'_i= \left\{ %
\begin{array} {l@{\qquad}l} m/|\Xi| &
\mbox{if } i\in\Xi,
\\
0 & \mbox{otherwise}, \end{array} %
\rright. \qquad i=1,\ldots,n .
\]
Begin by noticing that
\[
\sup_{\vec{b}\in\R_0^+: \sum_{i=1}^n b_i=m} \min_{S\in\C} \sum
_{i\in S} b_i \geq\min_{S\in\C} \sum
_{i\in S} b'_i=
\frac{ms}{|\Xi
|} .
\]
The proof proceeds by contradiction, and makes use of a probabilistic
argument. Suppose there is a vector $\vec{b}^*\in\R_0^+$ such that
$\sum_{i=1}^n b^*_i\leq m$ and
%
\begin{equation}\label
{eqn:contradiction}
\min_{S\in\C} \sum_{i\in S}
b^*_i>\frac{ms}{|\Xi|} .
\end{equation}
We show next that this in contradiction with the symmetry assumption.

Let $J$ be a uniform random variable with range $\Xi$. Then
%
\begin{equation}\label{eqn:E_J}
\E\bigl[b^*_J\bigr]=\frac{1}{|\Xi|}\sum
_{j\in\Xi} b^*_j\leq\frac{1}{|\Xi
|}\sum
_{j=1}^n b^*_j \leq\frac{m}{|\Xi|}
.
\end{equation}
Now construct another random variable $K$ in a hierarchical fashion:
first take $S$ drawn uniformly over $\C$, and given $S$ take $K$ drawn
uniformly over $S$.
Then clearly
%
\begin{eqnarray}\label{eqn:E_K}
\E\bigl[b^*_K\bigr] &=& \E\bigl[\E \bigl[
b^*_K \rrvert S \bigr] \bigr]
\nonumber
\\
&=& \E \biggl[\frac{1}{s}\sum_{k\in S}
b^*_k \biggr]
\nonumber
\\
&\geq& \E \biggl[\min_{S\in\C} \frac{1}{s}\sum
_{k\in S} b^*_k \biggr]
\\
&=& \frac{1}{s} \E \biggl[\min_{S\in\C} \sum
_{k\in S} b^*_k \biggr]
\nonumber
\\
&>& \frac{m}{|\Xi|} ,\nonumber
\end{eqnarray}
where the strict inequality follows from \eqref{eqn:contradiction}. To
conclude the proof, we just need to notice that $J$ and $K$ have
exactly the same distribution if the class $\C$ is symmetric. Let
$k\in\Xi$ be arbitrary. Then
\begin{eqnarray*}
\P(K=k) &=& \E\bigl[\1\{K=k\}\bigr]
\\
&=& \E\bigl[\E\bigl[\1\{K=k\} | S \bigr]\bigr]
\\
&=& \E \biggl[\frac{1}{s} \1\{k\in S\} \biggr]
\\
&=& \frac{1}{s} \P(k\in S)
\\
&=& \frac{1}{s} \frac{s}{|\Xi|}=\frac{1}{|\Xi|} .
\end{eqnarray*}
Therefore, both $J$ and $K$ are uniformly distributed over $\Xi$ and
so $\E[b^*_J]=\E[b^*_K]$. This creates a contradiction between \eqref
{eqn:E_J} and \eqref{eqn:E_K} invalidating the existence of vector
$\vec{b}^*$, concluding the proof.

For the second result, note simply that
\begin{eqnarray*}
\frac{1}{|\C|} \sum_{S\in\C} \sum
_{i\in S} b_i &=& \frac{1}{|\C
|} \sum
_{S\in\C} \sum_{i=1}^n
b_i \1\{i\in S\}
\\
&=& \sum_{i=1}^n b_i
\frac{1}{|\C|} \sum_{S\in\C} \1\{i\in S\}
\\
&=& \sum_{i=1}^n b_i
\frac{s}{|\Xi|} ,
\end{eqnarray*}
where the last step follows from the symmetry assumption. The result of
the lemma is now immediate.
\end{pf*}
%
%
\begin{pf*}{Proof of Proposition~\ref{prop:detection_other_risks}}
If $\tilde R(\hat\Phi)<\epsilon/2$, the result follows immediately
from the simple fact that $R(\hat\Phi) \leq2\tilde R(\hat\Phi)$.
Therefore, $\tilde R(\hat\Phi)<\epsilon/2$ implies that $R(\hat\Phi
)<\epsilon$ and we just apply the result of the theorem. For the
second statement, it is useful to look at $S$ as a uniform random
variable with range $\C$. In the proof of Theorem~\ref
{thm:detection}, we showed that, for any $S\in\C$
\[
\E_\emptyset[\log\LR_{\emptyset,S}|S] \geq-\log \bigl(2\P
_\emptyset(\hat\Phi\neq0)+2\P_1(\hat\Phi\neq1|S) \bigr)
,
\]
where $\P_1$ denotes the probability measure under the alternative
hypothesis. By taking the expectation on both sides, we have
\[
\E_\emptyset[\log\LR_{\emptyset,S}] \geq-\frac{1}{|\C|} \sum
_{S\in\C} \log \bigl(2\P_\emptyset(\hat\Phi\neq0)+2
\P_1(\hat \Phi\neq1|S) \bigr) .
\]
To simplify the notation, let $p_0\equiv\P_\emptyset(\hat\Phi\neq
0)$ and $p_S\equiv\P_1(\hat\Phi\neq1|S)$. The statement $\bar
R(\hat\Phi)\leq\epsilon$ is equivalent to $p_0+\frac{1}{|\C|}\sum_{S\in\C} p_S\leq\epsilon$. Accordingly define the constraint set
$\mathcal{P}\subseteq\R^{1+|\C|}$ as
\[
\mathcal{P}= \biggl\{p_0,\{p_S\}_{S\in\C}\dvtx
p_0+\frac{1}{|\C|}\sum_{S\in\C}
p_S\leq\epsilon \biggr\} .
\]
We have that\vspace*{-1pt}
%
\begin{eqnarray}
\E_\emptyset[\log\LR_{\emptyset,S}] &\geq& \min_{\mathcal{P}}
\biggl\{-\frac{1}{|\C|} \sum_{S\in\C} \log
(2p_0+2p_S ) \biggr\}
\\
&=& \log\frac{1}{2\epsilon} ,
\end{eqnarray}
where the last step follows from a straightforward Lagrange multiplier
argument, to conclude that the minimum is attained by taking
$p_0+p_S=\epsilon$ for all $S\in\C$.

The next step, similar to the proof of Theorem~\ref{thm:detection}, is
to solve $\sup_{\mathcal{A}} \E_\emptyset[\log\LR_{\emptyset
,S}]$, where it is important to recall that $S$ is random. Following
the same approach as in the proof of the theorem yields
\begin{eqnarray*}
\sup_{\mathcal{A}} \E_\emptyset[\log\LR_{\emptyset,S}] &=&
\frac
{\mu^2}{2} \sup_{\vec{b}\in\R_0^+: \sum_{i=1}^n b_i=n} \frac
{1}{|\C|}\sum
_{S\in\C} \sum_{i\in S} b_i
,
\end{eqnarray*}
where $b_i$ is defined in \eqref{eqn:b}. The second result of Lemma~\ref{lemma:optimization} characterizes the solution of this
optimization problem, and therefore
\[
\frac{\mu^2 m s}{2|\Xi|}\geq\log\frac{1}{2\epsilon} .
\]
Simple algebraic manipulation concludes the proof.
\end{pf*}
\begin{pf*}{Proof of Lemma~\ref{lemma:symmetrization}}
To ease the notation let $\C_s$ denote the class of all subsets of $\{
1,\ldots,n\}$ with cardinality $s$. Let $S\in\C_s$ and $i\in S$ be
fixed, but arbitrary. Note that the permutation $\perm$ maps this set
to another set $S^{(\perm)}=\perm(S)\in\C_s$ with the same
cardinality. Furthermore, since the permutation is chosen uniformly
over the set of all permutations this set is uniformly distributed over
$\C_s$, that is
\[
S^{(\perm)}\sim\Unif(\C_s) .
\]
In addition define the random variable $J=\perm(i)$. This is obviously
uniformly distributed over $\{1,\ldots,n\}$. More importantly,
conditionally on $S^{(\perm)}$, $J$ is uniformly distributed over the
set $S^{(\perm)}$. In other words, for arbitrary $k\in\{1,\ldots,n\}$
\begin{eqnarray*}
\P\bigl(J=k|S^{(\perm)}\bigr) &=& \P\bigl(\perm(i)=k|S^{(\perm)}\bigr)
\\
&=& \P\bigl(\perm^{-1}(k)=i|S^{(\perm)}\bigr)
\\
&=& \left\{ %
\begin{array} {l@{\qquad}l} 1/s & \mbox{if } k\in S^{(\perm)},
\\
0 & \mbox{otherwise}. \end{array} %
\rright.
\end{eqnarray*}
Therefore
\begin{eqnarray*}
\P\bigl(\hat S^{(\perm)}_i\neq1\bigr) &=& \E \bigl[\1\{\hat
S_{\perm(i)}\neq 1\} \bigr]
\\
&=& \E\bigl[\E \bigl[ \1\{ \hat S_{\perm(i)}\neq1\}\rrvert
S^{(\perm)} \bigr] \bigr]
\\
&=& \E \biggl[\frac{1}{s}\sum_{j\in S^{(\perm)}}
\P_{S^{(\perm)}} (\hat S_j\neq1) \biggr]
\\
&=& \frac{1}{|\C_s|}\sum_{S'\in\C_s} \frac{1}{s}
\sum_{j\in S'} \P_{S'} (\hat S_j
\neq1) ,
\end{eqnarray*}
where the two last steps follow from the distribution of $S^{(\perm)}$
and $\perm(i)$. The case $i\notin S$ is entirely analogous. Using
these two results we obtain the first two statements of the lemma for
the class $\C'=\C_{s-1}\cup\C_{s}\cup\C_{s+1}$. Finally, the last
result in the lemma follows trivially from the other two statements.
\end{pf*}
\begin{pf*}{Sketch proof of Corollary~\ref{coro:estimation_FDR}}
The result in the corollary follows in the same manner as the result in
Theorem~\ref{thm:estimation}, but noticing that for symmetric
estimation procedures the requirements on the estimator $\hat S_i$ for
each $i\in\{1,\ldots,n\}$ are much less stringent. In particular let
$S\in\C'$ be arbitrary and assume that
\[
R_{\FDR+\NDR}(\hat S,S) \leq\epsilon,
\]
where $\epsilon>0$, which implies that both FDR and NDR are less than
$\epsilon$. Now consider symmetric procedures and let $\alpha=\P
_S(\hat S_i\neq0)$ for $i\notin S$ and $\beta=\P_S(\hat S_i\neq1)$
for $i\in S$. Clearly, the constraint in NDR implies that
\[
\epsilon\geq\NDR(\hat S,S)=\E \biggl[\frac{|S \setminus\hat
S|}{|S|} \biggr]=
\frac{|S|\beta}{|S|}=\beta.
\]
The constraint on FDR is a bit more difficult to analyze, due to the
random denominator its definition. However, a very sloppy bound
suffices, namely
\[
\epsilon\geq\FDR(\hat S,S) = \E \biggl[\frac{|\hat S\setminus
S|}{|\hat S|} \biggr] \geq\E \biggl[
\frac{|\hat S\cap S^c|}{n} \biggr] = \frac{(n-|S|)\alpha}{n} .
\]
Therefore, we conclude that $\alpha\leq\epsilon$ suffices. Note that
this is a very loose but nevertheless sufficient bound. The rest of the
proof proceeds now in the same fashion as Theorem~\ref{thm:estimation}
and Corollary~\ref{coro:estimation}.
\end{pf*}
\begin{pf*}{Proof of Proposition~\ref{prop:CS_lowerbound}}
The proof of this result mimics closely the proof of Theorem~\ref
{thm:estimation}, with the necessary changes to account for the
different sensing model. The first step is to reduce the class of
signals under consideration. Clearly signals of the form \eqref
{eqn:simple_model} are also in the class $\mathcal{H}(\mu)$. Therefore
\[
\max_{\vec{x}\in\mathcal{H}{\mu}} \E_{\vec{x}}\bigl[d(\hat S,S)\bigr]\geq \max
_{S\in\C'} \E_S\bigl[d(\hat S,S)\bigr] ,
\]
where the expectation on the right-hand-side is taken assuming $\vec
{x}$ is of the form \eqref{eqn:simple_model} with support~$S$.
Condition \eqref{eqn:CS_condition} therefore implies that
\[
\max_{S\in\C'} \E_S\bigl[d(\hat S,S)\bigr]\leq
\epsilon,
\]
so, for the purpose of computing a lower bound it suffices to consider
on the signals where all the non-zero components are valued $\mu$. It
is important to note that this subclass of signals might not correspond
to the ``hardest'' signals to estimate, and no claim is made about
this. However, this subclass seems to capture the essential aspects of
the problem in light of the bounds derived. As the class of signals
under consideration is the same as in Theorem~\ref{thm:estimation} the
only change in that proof stems from the different observation model,
which in turn results in a different log-likelihood ratio. Notice that,
as before, we can consider only symmetric procedures in the sense of
Lemma~\ref{lemma:symmetrization}.

To aid in the presentation, let $A_{ij}$ denote the entry in the $i$th
row and $j$th column of the matrix $\vec{A}$, and let $\vec
{A}_{i\cdot}$ and $\vec{A}_{\cdot j}$ denote respectively the $i$th
row of and the $j$th column of $\vec{A}$. The log-likelihood ratio is
therefore given by
\begin{eqnarray*}
\log\LR_{S,S'}(\vec{Y},\vec{A}) &=& \log\frac{f(\vec{Y},\vec
{A};S)}{f(\vec{Y},\vec{A};S')}
\\
&=& \sum_{k=1}^\ell\log\frac{f_{Y_k|\vec{A}_{k\cdot}}(Y_k|\vec
{A}_{k\cdot};S)}{f_{Y_k|\vec{A}_{k\cdot}}(Y_k|\vec{A}_{k\cdot
};S')}
\\
&=& \frac{1}{2}\sum_{k=1}^\ell
\biggl[ \biggl(Y_k-\mu\sum_{j\in S'}
A_{kj} \biggr)^2- \biggl(Y_k-\mu\sum
_{j\in S} A_{kj} \biggr)^2 \biggr] .
\end{eqnarray*}
Given this, the expected log-likelihood ratio can be computed quite
easily as before, and we get
%
\begin{equation}\label{eqn:CS_LR}
\E_S \bigl[\log\LR_{S,S'}(\vec{Y},\vec{A}) \bigr] =
\frac{\mu^2}{2} \sum_{k=1}^\ell
\E_S \biggl[ \biggl( \biggl(\sum_{j\in S}
A_{kj} \biggr)- \biggl(\sum_{j\in S'}
A_{kj} \biggr) \biggr)^2 \biggr] .
\end{equation}
Now consider the sets $S^{(i)}$ as in the proof of Theorem~\ref
{thm:estimation}. Since we have $S\Delta S^{(i)}=\{i\}$, we get from
equation \eqref{eqn:CS_LR}
%
\begin{eqnarray}\label{eqn:CS_LR_i}
\E_S \bigl[\log\LR_{S,S^{(i)}}(\vec{Y},\vec{A}) \bigr] &=&
\frac
{\mu^2}{2} \E \Biggl[\sum_{k=1}^\ell
A_{ki}^2 \Biggr]
\nonumber
\\[-8pt]\\[-8pt]
&=& \frac{\mu^2}{2} \E \bigl[\|\vec{A}_{\cdot i}\|_F^2
\bigr] .\nonumber
\end{eqnarray}
From this point on, the proof proceeds in exactly the same fashion as
that of Theorem~\ref{thm:estimation}. Begin by summing the terms
\eqref{eqn:CS_LR_i} over $i\in\{1,\ldots,n\}$ to get an upper bound
on the expected likelihood ratio
%
\begin{equation}\label{eqn:CS_upper_bound}
\sum_{i=1}^n \E_S [ \log
\LR_{S,S^{(i)}} ] = \frac{\mu
^2}{2} \E \Biggl[\sum
_{i=1}^n \|\vec{A}_{\cdot i}
\|_F^2 \Biggr] = \frac{\mu^2}{2} \E \bigl[\|\vec{A}
\|_F^2 \bigr] \leq\frac{m \mu
^2}{2} .
\end{equation}
Finally, the lower bounds on the log-likelihood ratio in \eqref
{eqn:LR_lowerbound1} and \eqref{eqn:LR_lowerbound2} are not dependent
on the nature of the likelihood ratio itself, but rather on the desired
risk performance. So these bounds are valid in the compressed sensing
setting as well. As in the proof of Theorem~\ref{thm:estimation},
using these lower bounds together with \eqref{eqn:CS_upper_bound}
concludes the proof.
\end{pf*}

\begin{pf*}{Proof of Proposition~\ref{prop:estimation_upperbound}}
We begin by introducing an algorithm that achieves the desired
performance bound. Algorithm \ref{alg:SDS} is described here for
convenience of presentation and explained in detail in the next
paragraphs. It is essentially the algorithm presented in Malloy and Nowak \cite{malloy:11b} for the case of Gaussian observation noise.
\begin{algorithm}[t]
\caption{Simple distilled sensing.}\label{alg:SDS}
\SetKw{KwParameters}{Parameters:}
\SetKw{KwInitialization}{Initialization:}
\SetKw{KwReturn}{Terminate:}
\KwParameters{Number of steps $l$ and per-measurement precision $p$}\\
\KwInitialization{}\\
\qquad $ k\leftarrow0$, $i \leftarrow1$, $\hat S\leftarrow\emptyset$\\
\qquad $ c_i\leftarrow0$ for $i=1,\ldots,n$\\
\qquad $ \Gamma^2_j\leftarrow p$ for $j=1,2,\ldots$\\
\For{$i\leftarrow1$ \KwTo$n$}
{
\Repeat{$c_i=l$ or $Y_k<0$}
{
$k\leftarrow k+1$\\
$c_i\leftarrow c_i+1$\\
Measure $Y_k\equiv Y_i^{(c_i)}=x_i+\Gamma_k^{-1} W_k$\\
\If{$p(k+1)>m$}{\KwReturn{Output $\hat S$}}
}
\If{$c_i=l$ and $Y_k\geq0$}{$\hat S \leftarrow\hat S \cup\{i\}$}
}
\KwReturn{Output $\hat S$}
\end{algorithm}

Sensing is performed coordinate-wise in a sequential way, until all the
signal entries have been explored or the total sensing budget is
exhausted. Note that all the measurements are made with the same
precision $p$. For each signal entry $i$ the algorithm performs at most
$l$ measurements. If any of these measurements is negative then entry
$i$ is deemed not to belong to the support estimate~$\hat S$. If all
the $l$ measurements are non-negative, then entry $i$ is deemed to
belong to the support estimate. For convenience, we identify the
measurements of entry $i$ by $Y_i^{(j)}$, where $j\in\{1,\ldots,l\}$.

In a sense, the algorithm is a very crude version of a sequential
likelihood ratio test. Given that we are interested in the general
rates of error decay we do not optimize the algorithm parameters for
performance and instead make crude choices that are sufficient to prove
the result. In particular we take $p=m/(4n)$ and $l=\log_2^2 n$.

The proof goes by showing first that, with high probability, the
algorithm terminates before reaching the total sensing budget.
Therefore, for the analysis we consider a modification of the algorithm
were termination upon the event $p(k+1)>m$ is removed. Note that the
number of measurements collected for entry $i$ is simply $c_i$. These
are independent random variables. The total number of measurements
collected is $\sum_{i=1}^n c_i$. Note that for all $i$ we have $0\leq
c_i\leq l$. Furthermore, note that for $i\notin S$ the corresponding
measurements $Y_i^{j}$ are zero mean normal random variables, which
means that $\P_S(Y_i^{(j)}<0)=1/2$. Therefore, $c_i$ corresponds to a
truncated geometric random variable:
\[
i\notin S,\qquad  \P_S(c_i=x)= \left\{ %
\begin{array}
{l@{\qquad}l} (1/2)^x & \mbox{if } x=1,\ldots,l-1,
\\
(1/2)^{l-1} & \mbox{if } x=l,
\\
0 & \mbox{otherwise}. \end{array} %
\rright.
\] %
Since these are truncated geometric random variables it is
clear that $\E_S(c_i)\leq2$ and $\V_S(c_i)
\leq2$. Now, Bernstein's inequality (as stated in Wasserman
\cite{wasserman:06}, page 9) tells us immediately that %
\[
\P_S \biggl(\sum_{i\notin S} c_i
-2(n-s) \geq t \biggr)\leq\exp \biggl(-\frac{1}{2}\frac{t^2}{2(n-s)+lt/3}
\biggr) .
\]
Taking $t=n-s$, and noting that $\sum_{i=1}^n c_i \leq sl+\sum_{i\notin S} c_i$ we conclude that
%
\begin{equation}
\P_S \Biggl(\sum_{i=1}^n
c_i < 3(n-s) +sl \Biggr)\geq1-\exp \biggl(-\frac{1}{2}
\frac{n-s}{2+l/3} \biggr) .
\end{equation}
Now, provided $s\leq n/(l-3)$, we conclude that the total number of
measurements of the algorithm is smaller than $4n$ with probability
approaching 1 as $n$ grows, that is
%
\begin{equation}
\P_S \Biggl(\sum_{i=1}^n
c_i < 4n \Biggr)\geq1-\exp \biggl(-\frac
{1}{2}
\frac{n-s}{2+l/3} \biggr) .
\end{equation}
Therefore the total amount of precision used is under $4np$ with high
probability. For the choice $p=m/(4n)$, the total amount of precision
used is less than $m$ with high probability. In other words,
%
\begin{equation}\label{eqn:HP}
\P_S\bigl(p(k+1)> m\bigr)\leq\exp \biggl(-\frac{1}{2}
\frac{n-s}{2+l/3} \biggr) .
\end{equation}

This result ensures the modified algorithm is essentially the same as
the original one, as in the latter we will rarely encounter the event
$p(k+1)> m$ (this statement will be made precise later). Therefore, we
can proceed by analyzing the performance of the modified algorithm.
This can be done in a entry-wise fashion and we must consider the cases
$i\in S$ and $i \notin S$. For $i\notin S$, note that
\[
\P_S(i\in\hat S)=\P_S \Biggl(\bigcap
_{i=1}^l \bigl\{Y_i^{(j)}\geq0
\bigr\} \Biggr)=\frac{1}{2^l} .
\]
For $i\in S$, we have
\[
\P_S(i\notin\hat S)\leq\P_S \Biggl(\bigcup
_{i=1}^l \bigl\{Y_i^{(j)}<0
\bigr\} \Biggr)=\frac{l}{2}\exp \biggl(-\frac{p\mu^2}{2} \biggr) ,
\]
where the result follows from a Gaussian tail and the union (of events)
bounds. These two results together give
\begin{eqnarray*}
\E_S\bigl[d(\hat S,S)\bigr] &=& \sum_{i\notin S}
\P_S(i\in\hat S) + \sum_{i\in S}
\P_S(i\notin\hat S)
\\
&\leq& \frac{n-s}{2^l} + \frac{sl}{2}\exp \biggl(-\frac{p\mu
^2}{2}
\biggr)
\\
&\leq& \frac{n-s}{2^l} + \frac{1}{2}\exp \biggl(-\frac{p\mu
^2-2\log s-2\log l}{2}
\biggr) .
\end{eqnarray*}
Now, given the choice $l=\log^2_2 n$ we conclude that the first term
in the above summation converges to 0 as $n\rightarrow\infty$, and
the second term also converges to zero provided
\[
-p\mu^2-2\log s-2\log l\rightarrow\infty
\]
as $n\rightarrow\infty$. Clearly if $\mu\geq\sqrt{\frac
{4n}{m}(2\log s+5\log\log_2 n)}$ this condition is satisfied for $n$
large enough. To conclude the proof all that remains to be done is to
take equation \eqref{eqn:HP} into account to conclude that, for the
original algorithm
\begin{eqnarray*}
\E_S\bigl[d(\hat S,S)\bigr] &\leq& \E_S\bigl[d(\hat
S,S)|p(k+1)\leq m\bigr]+\E_S\bigl[d(\hat S,S)|p(k+1)> m\bigr]
\P_S\bigl(p(k+1)> m\bigr)
\\
&\leq& \E_S\bigl[d(\hat S,S)|p(k+1)\leq m\bigr]+n\P_S
\bigl(p(k+1)> m\bigr)
\\
&\leq& \frac{n-s}{2^l} + \frac{1}{2}\exp \biggl(-\frac{p\mu
^2-2\log s-2\log l}{2}
\biggr)+n\exp \biggl(-\frac{1}{2}\frac
{n-s}{2+\log^2_2 n/3} \biggr) .
\end{eqnarray*}
Clearly, under the condition $s\leq n/(l-3)$ all the terms above
converge to zero as $n\rightarrow\infty$, concluding the proof.
\end{pf*}

\end{appendix}
\section*{Acknowledgements}

 I  want to thank Nikhil Bansal for suggesting the
elegant proof of Lemma~\ref{lemma:optimization}. The modification of
DS proposed in Section~\ref{sec:tightness_detection} came into being
after discussions with Jarvis Haupt. Finally, the author wants to thank
the two anonymous referees and the associate editor for their valuable
comments and suggestions.



\printhistory

\end{document}